\documentclass[11pt]{amsart}
\usepackage{amssymb,amsthm,tikz}

\usetikzlibrary{matrix,arrows}

\newtheorem{theo}{Theorem}
\newtheorem{lemma}[theo]{Lemma}
\newtheorem{cor}[theo]{Corollary}
\newtheorem{proposition}[theo]{Proposition} 
\newtheorem{conjecture}[theo]{Conjecture}
\newtheorem{remark}[theo]{Remark}

\theoremstyle{definition}
\newtheorem{definition}[theo]{Definition}
\newtheorem{example}[theo]{Example}

\newcommand{\C}{\mathrm C}
\newcommand{\D}{\mathrm D}

\newcommand{\OO}{\mathrm O}
\newcommand{\V}{\mathrm V}
\newcommand{\Z}{\mathrm Z}

\newcommand{\FF}{\mathbb F}

\newcommand{\PGammaL}{\mathop{\mathrm{P\Gamma L}}}
\newcommand{\GL}{\mathop{\mathrm{GL}}}
\newcommand{\SL}{\mathop{\mathrm{SL}}}
\newcommand{\PGL}{\mathop{\mathrm{PGL}}}
\newcommand{\PSL}{\mathop{\mathrm{PSL}}}
\newcommand{\PSU}{\mathop{\mathrm{PSU}}}
\newcommand{\Ree}{\mathop{\mathrm{Ree}}}
\newcommand{\Suz}{\mathop{\mathrm{Suz}}}

\newcommand{\Sym}{\mathop{\mathrm{Sym}}}

\newcommand{\Cos}{\mathop{\mathrm{Cos}}}

\newcommand{\Aut}{\mathop{\mathrm{Aut}}}

\renewcommand{\P}{\mathcal P}
\renewcommand{\wr}{\mathop{\mathrm{wr}}}

\begin{document}
\title{On graph-restrictive permutation groups}

\author[P. Poto\v{c}nik]{Primo\v{z} Poto\v{c}nik}
\address{Primo\v{z} Poto\v{c}nik, Institute of Mathematics, Physics, and
  Mechanics, \newline 
Jadranska 19, 1000 Ljubljana, Slovenia}\email{primoz.potocnik@fmf.uni-lj.si}

\author[P. Spiga]{Pablo Spiga}
\address{Pablo Spiga,  School of Mathematics and Statistics,\newline
The University of Western Australia,
 Crawley, WA 6009, Australia} \email{spiga@maths.uwa.edu.au}

\author[G. Verret]{Gabriel Verret}
\address{Gabriel Verret, Institute of Mathematics, Physics, and
  Mechanics, \newline 
Jadranska 19, 1000 Ljubljana, Slovenia}
\email{gabriel.verret@fmf.uni-lj.si}

\thanks{The second author is supported by UWA as part of the
Australian Council Federation Fellowship Project FF0776186. The second author would also like to thank the Institute of
Mathematics, Physics and Mechanics at the University of Ljubljana for the
warm hospitality.}

\subjclass[2000]{Primary 20B25; Secondary 05E18}

\keywords{arc-transitive graphs, semiprimitive group, primitive group, graph-restrictive group, Weiss Conjecture} 

\begin{abstract}
Let $\Gamma$ be a connected $G$-vertex-transitive graph, let $v$ be a vertex of $\Gamma$ and let $L=G_v^{\Gamma(v)}$
be the permutation group induced by the action of the vertex-stabiliser $G_v$ on the neighbourhood $\Gamma(v)$.
Then $(\Gamma,G)$ is said to be \emph{locally-$L$}. A transitive permutation group $L$ is \emph{graph-restrictive} if there exists a constant
$c(L)$ such that, for every locally-$L$ pair $(\Gamma,G)$ and an arc $(u,v)$ of $\Gamma$, the
inequality $|G_{uv}|\leq c(L)$ holds.

Using this terminology, the Weiss Conjecture says that primitive groups are graph-restrictive. We propose a very strong generalisation of this conjecture: a group is graph-restrictive if and only if it is semiprimitive. (A transitive permutation group is said to be \emph{semiprimitive} if each of its normal subgroups is either transitive or semiregular.) Our main result is a proof of one of the two implications of this conjecture, namely that graph-restrictive groups are semiprimitive. We also collect the known results and prove some new ones regarding the other implication.
\end{abstract}

\maketitle

\section{Introduction}
Unless explicitly stated, all graphs considered in this paper are finite and simple. A graph $\Gamma$ is said to be $G$-\emph{vertex-transitive} if $G$ is a subgroup of $\Aut(\Gamma)$ acting transitively on the vertex-set $\V\Gamma$ of $\Gamma$. Similarly, $\Gamma$ is said to be $G$-\emph{arc-transitive} if $G$ acts transitively on the arcs of $\Gamma$ (that is, on the ordered pairs of adjacent vertices of $\Gamma$). When $G=\Aut(\Gamma)$, the prefix $G$ in the above notation is sometimes omitted.

One of the most important early results concerning arc-transitive graphs is the beautiful theorem of Tutte~\cite{Tutte, Tutte2}, saying that, in a connected $3$-valent arc-transitive graph, the order of an arc-stabiliser divides 16. The immediate generalisation of Tutte's result to 4-valent graphs is false; there exist connected $4$-valent arc-transitive graphs with arbitrarily large arc-stabilisers (for example, see~\cite{PSV}).

On the other hand, it can easily be deduced from the work of Gardiner~\cite{Gard} that, if $\Gamma$ is a connected 4-valent $G$-arc-transitive graph and $(u,v)$ is an arc of $\Gamma$, then $|G_{uv}|\leq 2^23^6$ unless $G_v^{\Gamma(v)}\cong \D_4$, where $G_v^{\Gamma(v)}$ denotes the permutation group induced by the action of $G_v$ on the neighbourhood $\Gamma(v)$. In view of Gardiner's results, bounds on the order of the arc-stabiliser are now usually considered in terms of the \emph{local action} $G_v^{\Gamma(v)}$ rather than simply the valency. This leads us to the following definitions.

\begin{definition} \label{def:locally}
Let $\Gamma$ be a connected $G$-vertex-transitive graph, let $v$ be a vertex of $\Gamma$ and let $L$ be a permutation group which is permutation isomorphic to $G_v^{\Gamma(v)}$. Then $(\Gamma,G)$ is said to be \emph{locally-$L$}. More generally, if $\P$ is a  permutation group property, then $(\Gamma,G)$ will be called {\em locally-$\P$} provided that $G_v^{\Gamma(v)}$ possesses the property $\P$.
\end{definition}

Note that, if $\Gamma$ has valency $d$, then the permutation group $G_v^{\Gamma(v)}$ has degree $d$ and, up to permutation isomorphism, does not depend on the choice of $v$. In~\cite{Verret}, the third author introduced the following notion.

\begin{definition}\label{def:Restrictive} 
A transitive permutation group $L$ is \emph{graph-restrictive} if there exists a constant
$c(L)$ such that, for every locally-$L$ pair $(\Gamma,G)$ and an arc $(u,v)$ of $\Gamma$, the
inequality $|G_{uv}|\leq c(L)$ holds.
\end{definition} 

The above definition makes the statement of many results and questions very succinct. For example, Tutte's theorem can be restated as follows: the symmetric group $\Sym(3)$ in its natural action on $3$ points is graph-restrictive and the constant $c(\Sym(3))$ can be chosen to be $16$. Several authors generalised Tutte's result from $\Sym(3)$ to other primitive groups, which eventually led Weiss to pose the following conjecture.\medskip

\noindent\textbf{Weiss Conjecture}~\cite[Conjecture 3.12]{Weiss}.
\emph{Primitive groups are graph-restrictive.}\\

While much effort has been deployed in the attempts to prove the Weiss Conjecture, very few authors considered graph-restrictiveness of imprimitive groups. We will show that the main results which have been used to attack the Weiss Conjecture can be generalised from primitive groups to semiprimitive groups, which we conjecture are graph-restrictive. (A transitive permutation group is said to be \emph{semiprimitive} if each of its normal subgroups is either transitive or semiregular. Contrary to~\cite{BerMar}, we consider regular groups to be semiprimitive).  In fact, encouraged by some results proved in this paper, we conjecture the following characterisation of graph-restrictive groups.

\begin{conjecture}\label{Conj:SP}
A permutation group is graph-restrictive if and only if it is semiprimitive.
\end{conjecture}

Our goal in this paper is to collect the known results regarding this conjecture and to establish new ones. One of our main results is the following theorem (proved in Section~\ref{Sec:ResAreSP}), which proves one of the two implications of the above conjecture.

\begin{theo}\label{Theo:ResAreSP}
Every  graph-restrictive group is semiprimitive.
\end{theo}

The structure of the paper is very simple. Section~\ref{Sec:Summary} is a summary of all important results, with each new theorem getting its own section later for details and proofs.

\section{Summary}
\label{Sec:Summary}
We start with a review of previously known results about graph-restrictiveness of permutation groups. It is a rather trivial observation that regular permutation groups $L$ are graph-restrictive, with $c(L)=1$. One of the earliest non-trivial results with important applications towards the Weiss Conjecture is Theorem~\ref{LocalGardiner}, a variant of the Thompson-Wielandt theorem (see \cite{VanBon}). Given a $G$-arc-transitive graph $\Gamma$ and $(u,v)$ an arc of $\Gamma$, denote by $G_{uv}^{[1]}$ the subgroup of $G$ fixing $\Gamma(u)$ and $\Gamma(v)$ point-wise.
\begin{theo}\label{LocalGardiner}\cite[Corollary 2.3]{Gard}
Let $(\Gamma,G)$ be a locally-primitive pair and let $(u,v)$ be an arc of $\Gamma$. Then, $G_{uv}^{[1]}$ is a $p$-group for some prime $p$.
\end{theo}

Theorem~\ref{LocalGardiner} imposes a very strong restriction on the structure of $G_{uv}$. For example, if $\Gamma$ has valency $d$, then $G_{uv}$ contains a normal $p$-subgroup of index at most $(d-1)!^2$.  As we will see, Theorem~\ref{LocalGardiner} was the first step in the proof of many results regarding the Weiss Conjecture.

Despite considerable effort by many authors, the Weiss Conjecture is still open. Some important subcases have however been dealt with, the most important being the 2-transitive case, which has been settled as the culmination of work by Trofimov and Weiss.

\begin{theo}\label{TwoTransitiveTheo}
$2$-transitive groups are graph-restrictive.
\end{theo}

The proof of Theorem~\ref{TwoTransitiveTheo} is scattered over many papers and hence this result is somewhat part of folklore. In Section~\ref{TwoTransitive}, we try to remedy this situation by collecting the various parts of the proofs and giving a brief overview of the argument. 

Another case of the Weiss Conjecture which was completed very recently is the case of primitive groups of so-called linear type. Let $q$ be a prime-power and let $V$ be a vector space over $\FF_q$ of dimension $n$. For $m\leq n$, denote by ${V \brack m}$ the set of all $m$-dimensional subspaces of $V$ and by $\PSL(n,q)^{V \brack m}$ the projective special linear group $\PSL(n,q)$ in its natural action on ${V \brack m}$. A primitive group $L$ is said to be of \emph{linear type} if it contains a normal subgroup isomorphic to $\PSL(n,q)^{V \brack m}$. In \cite{TrofSup,TrofWeiss,TrofWeiss2,WeissGrass}, it was proved that primitive groups of linear type are graph-restrictive.

Another important case of the Weiss Conjecture is the case of primitive groups of \emph{affine type}, which is almost complete (see~\cite[Theorem]{weissp}). (A primitive group $L$ is said to be of affine type if it contains a regular abelian normal subgroup $Q$. In particular, as $L$ is primitive, $Q$ is an elementary abelian $q$-group, for some prime $q$, and $L$ has degree a power of $q$.) 

By Burnside's Theorem~\cite[Theorem 3.5B]{DixMor}, transitive groups of prime degree are either 2-transitive or of affine type. Together with Tutte's result,~\cite[Theorem]{weissp} and Theorem~\ref{TwoTransitiveTheo} imply that transitive groups of prime degree are graph-restrictive.

Note that some of the above results were misquoted in \cite{ConderLiPraeger}. For example, \cite[Theorem]{weissp} was misquoted in~\cite[Theorem 3.1]{ConderLiPraeger} where the authors concluded that all primitive groups of affine type are graph-restrictive, which does not follow from~\cite[Theorem]{weissp}. This has unfortunate consequences. For example, they claim that primitive groups of degree at most 20 are graph-restrictive~\cite[Proposition 4.1]{ConderLiPraeger}, but their proof relies on~\cite[Theorem 3.1]{ConderLiPraeger} in an essential way. Moreover, in their proof, they also claim that almost simple groups with socle isomorphic to $\PSL(n,q)$ for some integer $n\neq 5$ are graph-restrictive but, again, this does not seem to follow from the papers they quote. In fact, to the best of our knowledge, it is still unknown whether the primitive group of affine type $\Sym(3)\wr\Sym(2)$ with its product action of degree $9$ is graph-restrictive. Similarly, it is also unknown whether the primitive group of almost simple type $\Sym(5)$ acting on the $10$ unordered pairs of a $5$-set is restrictive. We will return to the question of graph-restrictiveness of groups of small degree in Section~\ref{subsec}.

Little was previously known about graph-restrictiveness of imprimitive groups. Recently, Sami~\cite{Sami} has shown that dihedral groups of odd degree are graph-restrictive, generalising Tutte's theorem (as $\Sym(3)\cong\D_3$), and the third author generalised this result to the wider class of so-called $p$-sub-regular groups~\cite[Theorem 1.2]{Verret}. (Of course, in view of Theorem~\ref{Theo:ResAreSP}, all these groups are semiprimitive.) To the best of our knowledge, this concludes the list of permutation groups that were known to be graph-restrictive prior to this paper.

Praeger has proved~\cite{PConj} that a quasiprimitive group $L$ is graph-restrictive if and only if there exists a constant $c'(L)$ such that, for every locally-$L$ pair $(\Gamma,G)$ with $G$ quasiprimitive or biquasiprimitive and an arc $(u, v)$ of $\Gamma$, the inequality $|G_{uv}|\leq c'(L)$ holds. (A transitive permutation group is called \emph{quasiprimitive} if each of its non-trivial normal subgroups is transitive. It is called \emph{biquasiprimitive} if it is not quasiprimitive and each of its non-trivial normal subgroups has at most two orbits.) This lead her to conjecture that quasiprimitive groups are graph-restrictive, a conjecture stronger than the Weiss Conjecture but weaker than Conjecture~\ref{Conj:SP}.

Let us now discuss our new results supporting the conjecture that semiprimitive groups are graph-restrictive. It turns out that many of the important tools that are available for primitive groups are also available for semiprimitive groups. For example, the starting point for most of the results in the primitive case is Theorem~\ref{LocalGardiner}, which was recently generalised to the semiprimitive case by the second author.

\begin{theo}\label{LocalSpiga}\cite[Corollary~$3$]{Spiga}
Let $(\Gamma,G)$ be a locally-semiprimitive pair and let $(u,v)$ be an arc of $\Gamma$. Then, $G_{uv}^{[1]}$ is a $p$-group for some prime $p$.
\end{theo}

Theorem~\ref{LocalSpiga} is very useful to prove that certain semiprimitive permutation groups are graph-restrictive. We will give some examples of this, but first we must introduce the following more general version of Definition~\ref{def:Restrictive}.

\begin{definition}
Let $p$ be a prime. A transitive permutation group $L$ is \emph{$p$-graph-restrictive} if there exists a constant
$c(p,L)$ such that, for every locally-$L$ pair $(\Gamma,G)$, the largest $p$-th power dividing the order of an arc-stabiliser is bounded above by $c(p,L)$.
\end{definition}

Note that a transitive permutation group $L$ is graph-restrictive if and only if it is $p$-graph-restrictive for every prime $p$ dividing the order of a point-stabiliser $L_x$. Theorem~\ref{LocalSpiga} has the following easy corollary, which will be proved in Section~\ref{TwoTransitive}. (Recall that, for a prime $p$, $\OO_p(G)$ denotes the largest normal $p$-subgroup of $G$.)

\begin{cor}\label{TechnicalPLocal2}
Let $L$ be a semiprimitive group acting on $\Omega$, let $x\in\Omega$ and let $p$ be a prime. If $\OO_p(L_x)=1$, then $L$ is $p$-graph-restrictive.
\end{cor}

In Section~\ref{p-restrictive}, we will use a result of Glauberman about normalisers of $p$-groups to prove Corollary~\ref{Cor:p-restrictive}, which is a generalisation of~\cite[Theorem 1.2]{Verret}.

\begin{cor}
\label{Cor:p-restrictive}
Let $p$ be a prime and let $L$ be a transitive permutation group on $\Omega$. Let $x\in\Omega$ and let $P$ be a Sylow $p$-subgroup of $L_x$. Suppose that 
\begin{enumerate}
\item $|P|=p$, and 
\item there exists $l\in L$ such that $\langle P,P^l\rangle$ is transitive on $\Omega$.
\end{enumerate}
Then, $L$ is $p$-graph-restrictive. In fact, we can take $c(p,L)=p^6$ if $p$ is odd and $c(p,L)=16$ if $p=2$.
\end{cor}

\section{Examples and groups of small degrees}\label{subsec}
In this section, we give a few examples of how Corollary~\ref{TechnicalPLocal2} and Corollary~\ref{Cor:p-restrictive} can be combined to show that certain semiprimitive permutation groups are graph-restrictive and examine the status of Conjecture~\ref{Conj:SP} for groups of small degree.

\begin{example}
Let $n\geq 3$ be odd and let $L=\D_n$ be the dihedral group of order $2n$ in its natural action on $n$ points. Then $|L_x|=2$ and, if $l$ is a generator of the cyclic subgroup $\C_n\leq L$, then $\langle L_x,L_x^l\rangle=L$. By Corollary~\ref{Cor:p-restrictive}, it follows that $L$ is graph-restrictive (in fact, we can take $c(L)=16$). 
\end{example}

\begin{example}
Let $p$ be a prime. Denote by $Z$ the centre of $\GL(2,p)$, that is, the subgroup of $\GL(2,p)$ consisting of the scalar matrices. Let $G$ be a subgroup of $\GL(2,p)$ with $\SL(2,p)\leq G$, and let $K$ be a subgroup of $G\cap Z$. Clearly, $\GL(2,p)$ acts as a group of automorphisms on the $2$-dimensional vector space of row vectors $\mathbb{F}_p^2$. We let $\Omega$ denote the set of orbits of $K$ on $\mathbb{F}_p^{2}\setminus \{0\}$. Since $K\unlhd G$, the group $G$ acts on $\Omega$ and, by a direct computation, we see that $K$ is the kernel of the action of $G$ on $\Omega$. Write $L=G/K$ for the permutation group induced by $G$ on $\Omega$. 

\smallskip

\noindent\textsc{Claim. }The group $L$ is graph-restrictive.

\smallskip

\noindent 
We prove this claim as follows. Write $r=|G:\SL(2,p)|$. We let $e_1=(1, 0)$, $e_2=(0, 1)$, $\alpha=e_1^K$ and $\beta=e_2^K$. Since $\SL(2,p)\leq G$, the group $G$ is transitive on $\mathbb{F}_p^{2}\setminus\{0\}$ and hence
\[
L_\alpha=\frac{G_{e_1}K}{K}\cong \frac{G_{e_1}}{K\cap G_{e_1}}\cong G_{e_1}=\left\{
\left(
\begin{array}{cc}
1&0\\
a&b\\
\end{array}
\right)\mid 
a,b\in \mathbb{F}_p,b^r=1
\right\}\cong C_p\rtimes C_r.
\]
It follows that $L_\alpha$ is a Frobenius group with Frobenius kernel of size $p$ and with Frobenius complement of size $r$. In particular, $\OO_q(L_\alpha)\neq 1$ if and only if $q=p$. Thus, by Corollary~\ref{TechnicalPLocal2}, $L$ is $q$-graph-restrictive for each prime $q\neq p$.

Finally, let 
\[
U_{e_1}=\left\{
\left(
\begin{array}{cc}
1&0\\
a&1
\end{array}
\right)
\mid a\in \mathbb{F}_p
\right\}
\textrm{ and }
U_{e_2}=\left\{
\left(
\begin{array}{cc}
1&a\\
0&1\\
\end{array}
\right)
\mid a\in \mathbb{F}_p
\right\}
\]
be the unipotent radicals of $G_{e_1}$ and $G_{e_2}$, and let $P_\alpha=U_{e_1}K/K$ and $P_\beta=U_{e_2}K/K$. We have $|P_\alpha|=|P_\beta|=p$ and $P_\alpha$ is a Sylow $p$-subgroup of $L_\alpha$. Furthermore, since $\SL(2,p)$ is generated by the root subgroups $U_{e_1}$ and $U_{e_2}$, we have that $\langle P_\alpha,P_\beta\rangle$  is transitive on $\Omega$. Thus, by Corollary~\ref{Cor:p-restrictive}, $L$ is $p$-graph-restrictive. Therefore $L$ is graph-restrictive.~$_\blacksquare$

Note that, in the extremal cases when $G=\SL(2,p)$ or $G=\GL(2,p)$ and $K=1$ or $K=G\cap Z$, then $L$ is simply one of $\GL(2,p)$, $\SL(2,p)$, $\PGL(2,p)$ or $\PSL(2,p)$ in their natural action.
\end{example}

Using a computer algebra system containing a list of transitive permutation groups of small degree (for example \texttt{Magma}~\cite{magma}), it is rather straightforward to generate an exhaustive list of semiprimitive groups of small degree. By going through such a list and applying a combination of \cite[Theorem]{weissp}, Theorem~\ref{TwoTransitiveTheo}, Corollary~\ref{TechnicalPLocal2} and Corollary~\ref{Cor:p-restrictive}, we obtain the following result:

\begin{proposition}
A semiprimitive permutation group of degree at most 13 is
graph-restrictive unless possibly it is one of the following:
\begin{enumerate}
\item $\Sym(3)\wr\Sym(2)$ with its product action of degree $9$,
\item $\mathbb{Z}_3^2\rtimes\mathbb{Z}_2$ of degree $9$ ($\mathbb{Z}_3^2$
acts regularly on itself by
multiplication and $\mathbb{Z}_2$ acts on $\mathbb{Z}_3^2$ by inversion),
\item $\Sym(5)$ of degree $10$ acting on the unordered pairs of a $5$-set, or
\item $\Sym(4)$ of degree $12$  acting on the ordered pairs of a $4$-set.
\end{enumerate}
\end{proposition}

In particular, Conjecture~\ref{Conj:SP} is true for permutation groups of degree at most 8. As we already noted, it is unknown whether the primitive group $\Sym(3)\wr\Sym(2)$ of degree 9 is graph-restrictive, hence even the Weiss Conjecture is not known to hold for groups of degree 9 or more.

\section{Proof of Theorem~\ref{Theo:ResAreSP}}\label{Sec:ResAreSP}

The proof of Theorem~\ref{Theo:ResAreSP} relies on a construction inspired by the \emph{wreath extension} (see~\cite[Section 8.1]{MaMa}), which has been used to solve some instances of the Embedding Galois Problem (see~\cite[Chapter 3]{IshFedLur}). We start by establishing some notation and some preliminary lemmas that will be necessary in the proof of Theorem~\ref{Theo:ResAreSP}.

Let $L$ be a finite permutation group on a set $\Lambda$ and let $K$ be an intransitive normal subgroup of $L$. Denote by $\Delta$ the set of orbits of $K$ on its action on $\Lambda$. Replacing $K$ by $\bigcap_{\lambda\in \Lambda}(KL_\lambda)$, we may assume that $K$ is the kernel of the action of $L$ on $\Delta$. We let $S$ denote the permutation group induced by $L$ on $\Delta$ and hence $S\cong L/K$. In particular, we have the short exact sequence

\begin{center}
\begin{tikzpicture}[descr/.style={fill=white,inner sep=2.5pt}]
\matrix (m) [matrix of math nodes, row sep=3em,
column sep=3em]
{1& K&L & S &1.\\};
\path[->,font=\scriptsize]
(m-1-1) edge node[auto] {} (m-1-2)
(m-1-2) edge node[auto] {} (m-1-3)
(m-1-3) edge node[auto] {$ \pi $} (m-1-4)
(m-1-4) edge node[auto] {} (m-1-5);
\end{tikzpicture}
\end{center}
Fix $\delta\in \Delta$ and $\lambda\in \delta$. Since $K$ is transitive on $\delta$, we have $L_\delta=KL_\lambda$, $S_\delta\cong L_\lambda/K_\lambda$ and the short exact sequence
\begin{center}
\begin{tikzpicture}[descr/.style={fill=white,inner sep=2.5pt}]
\matrix (m) [matrix of math nodes, row sep=3em,
column sep=3em]
{1& K_\lambda&L_\lambda & S_\delta &1\\};
\path[->,font=\scriptsize]
(m-1-1) edge node[auto] {} (m-1-2)
(m-1-2) edge node[auto] {} (m-1-3)
(m-1-3) edge node[auto] {$ \pi $} (m-1-4)
(m-1-4) edge node[auto] {} (m-1-5);
\end{tikzpicture}
\end{center}
(where we denote by $\pi$ the restriction  $\pi_{|L_\lambda}:L_\lambda\to S_\delta$).

Fix $T$ a transversal for the set of right cosets of $S_\delta$ in $S$. Without loss of generality, we assume that $1\in T$. Given $s\in S$, there exists a unique element of $T$, which we denote by  $s^\tau$, such that $S_\delta s=S_\delta s^\tau$. The correspondence $s\mapsto s^\tau$ defines a map $\tau:S\to T$ with $1^\tau=1$. 

\begin{lemma}\label{silly1}
If $x,s\in S$, then $(xs^{-1})^\tau s (x^\tau)^{-1}\in S_\delta$.
\end{lemma}
\begin{proof}
We have $S_\delta xs^{-1}=S_\delta(xs^{-1})^\tau$ and hence $S_\delta x=S_\delta (xs^{-1})^\tau s$. Furthermore, as $S_\delta x=S_\delta x^{\tau}$, we obtain $S_\delta=S_\delta (xs^{-1})^\tau s(x^{\tau})^{-1}$.
\end{proof}

Consider the set 
\begin{eqnarray*}
\Omega&=&\{f:S\to L_\lambda\mid f(yx)=f(x) \textrm{ for every }y\in S_\delta,x\in S\}.
\end{eqnarray*}
The elements of $\Omega$ are functions $f:S\to L_\lambda$ which (for each $x\in S$) are constant on the right coset $S_\delta x$ of $S_\delta$ in $S$, and hence they can be thought of as functions from $\Delta$ to $L_\lambda$.  The set $\Omega$ is a group isomorphic to $L_\lambda^{|\Delta|}$ under point-wise multiplication. Given $f\in \Omega$ and $g\in L$, let $f^g$ be the element of $\Omega$ defined by 
\begin{equation}\nonumber
f^g(x)=f(x(g^\pi)^{-1}).
\end{equation}
This defines a group action of $L$ on $\Omega$ and the semidirect product $\Omega\rtimes L$ is isomorphic to the standard wreath product $L_\lambda\wr_\Delta L$. Moreover, given an arbitrary positive integer $m$, by extending this action of $L$ on $\Omega$ to the (component-wise) action of $L$ on $\Omega^m$, we obtain a semidirect product $\Omega^m\rtimes L$ where the multiplication is given by 

\begin{eqnarray}\label{eq:binary}
(g,f_1,\ldots,f_m)(g',h_1,\ldots,h_m)&=&(gg',f_1^{g'}h_1,\ldots, f_m^{g'}h_m).
\end{eqnarray}

Consider the subset 

\begin{eqnarray}\label{A}
A=\{(g,f_1,\ldots,f_m)&\mid& g\in L, \textrm{ for each }i\in \{1,\ldots,m\}, f_i\in \Omega \textrm{ and,}\\\nonumber
&&\textrm{for every }x\in S, (f_i(x))^\pi=(x(g^\pi)^{-1})^\tau g^\pi(x^\tau)^{-1}\}
\end{eqnarray}
of $\Omega^m\rtimes L$. Note that, by Lemma~\ref{silly1}, the element $(x(g^\pi)^{-1})^\tau g^\pi(x^\tau)^{-1}$ in the definition of $A$ lies in $S_\delta$. 

\begin{lemma}\label{silly2}
The set $A$ is a subgroup of $\Omega^m\rtimes L$. 
\end{lemma}
\begin{proof}
Let $(g,f_1,\ldots,f_m),(g',h_1,\ldots,h_m)\in A$. For each $i\in \{1,\ldots,m\}$, we have $f_i^{g'}h_i\in \Omega$.  Fix $x\in S$. Then, for each $i\in \{1,\ldots,m\}$, using the definition of $A$, we obtain
\begin{eqnarray*}
((f_i^{g'}h_i)(x))^\pi&=&(f_i^{g'}(x))^\pi(h_i(x))^\pi=(f_i(x(g'^\pi)^{-1}))^\pi(h_i(x))^\pi\\
&=&(((x(g'^\pi)^{-1})(g^\pi)^{-1})^\tau g^\pi ((x(g'^\pi)^{-1})^\tau)^{-1})((x(g'^\pi)^{-1})^\tau g'^\pi(x^\tau)^{-1})\\\nonumber
&=&(x((gg')^\pi)^{-1})^\tau (gg')^\pi(x^{\tau})^{-1}.
\end{eqnarray*}
From~$(\ref{eq:binary})$ and~$(\ref{A})$, this shows that the product $(g,f_1,\ldots,f_m)(g',h_1,\ldots,h_m)$ lies in $A$. Denote by $e:S\to L_\lambda$ the function with $e(x)=1$ for every $x\in S$. Clearly $(1,e,\ldots,e)$ is the identity of $\Omega^m\rtimes S$ and lies in $A$. Finally, if $(g,f_1,\ldots,f_m)$ is in $A$, then $(g^{-1},(f_1^{-1})^{g^{-1}},\ldots,(f_m^{-1})^{g^{-1}})$ is the inverse of $(g,f_1,\ldots,f_m)$. Fix $x\in S$. For each $i\in \{1,\ldots,m\}$, using the definition of $A$, we obtain 
$$(((f_i^{-1})^{g^{-1}})(x))^\pi=((f_i(xg^\pi))^\pi)^{-1}=(x^\tau g^\pi ((xg^\pi)^{\tau})^{-1})^{-1}=(xg^\pi)^\tau (g^\pi)^{-1}(x^\tau)^{-1}.$$
Thus $(g^{-1},(f_1^{-1})^{g^{-1}},\ldots,(f_m^{-1})^{g^{-1}})$ lies in $A$. 
\end{proof}

Consider the map  
\begin{equation*}
\varphi:A\to L\textrm{ defined by }(g,f_1,\ldots,f_m)^\varphi=g.
\end{equation*}

\begin{lemma}\label{silly3}
$\varphi$ is a surjective homomorphism.
\end{lemma}
\begin{proof}
From~$(\ref{eq:binary})$, $\varphi$ is a homomorphism.  For each $s\in S_\delta$, fix an element of $L_\lambda$, which we denote by $s^\varepsilon$, with $(s^\varepsilon)^\pi=s$. Since $\pi:L_\lambda\to S_\delta$ is surjective, $\varepsilon:S_\delta\to L_\lambda$ is a well-defined mapping with $s^{\varepsilon\pi}=s$ for every $s\in S_\delta$.

Fix $g\in L$. Given $x\in S$, we see from Lemma~\ref{silly1} that $(x(g^\pi)^{-1})^{\tau}g^\pi(x^{\tau})^{-1}\in S_\delta$. Therefore, the function 
$f_{g}:S\to L_\lambda$ with $f_g(x)=((x(g^\pi)^{-1})^{\tau}g^\pi(x^{\tau})^{-1})^\varepsilon$  is well-defined. If $y\in S_\delta$ and $x\in S$, then $(yx)^\tau=x^{\tau}$ and $(yx(g^\pi)^{-1})^\tau=(x(g^\pi)^{-1})^\tau$. Therefore 
 $$f_{g}(yx)=((yx(g^\pi)^{-1})^\tau g^\pi ((yx)^\tau)^{-1})^\varepsilon=((x(g^\pi)^{-1})^\tau g^\pi(x^\tau)^{-1})^\varepsilon=f_g(x).$$ 
 Since this holds for arbitrary $y\in S_\delta$ and $x\in S$, it follows that $f_g\in\Omega$. Now $$f_g(x)^\pi=(((x(g^\pi)^{-1})^{\tau}g^\pi(x^{\tau})^{-1})^\varepsilon)^\pi=((x(g^\pi)^{-1})^{\tau}g^\pi(x^{\tau})^{-1})^{\varepsilon\pi}=(x(g^\pi)^{-1})^{\tau}g^\pi(x^{\tau})^{-1}.$$ Finally, this shows that $(g,f_g,\ldots,f_g)\in A$ and $g=(g,f_g,\ldots,f_g)^\varphi$, which proves that $\varphi$ is surjective. 
\end{proof}

Let $M$ be the kernel of $\varphi$.  Note that $A$ does not necessarily split over $M$ and hence $A$ is not necessarily the semidirect product of $M$ with a subgroup of $A$ isomorphic to $L$. It can be shown that if $L_\lambda$ splits over $K_\lambda$, then $A$ splits over $M$. If $(g,f_1,\ldots,f_m)\in M$, then $g=(g,f_1,\ldots,f_m)^\varphi=1$ and, for each $i\in \{1,\ldots,m\}$, we have $(f_i(x))^\pi=(x(g^\pi)^{-1})^{\tau}g^\pi(x^\tau)^{-1}=1$. Hence $f_i(x)\in K_\lambda$ for every $x\in S$. Therefore $$M=\{(1,f_1,\ldots,f_m)\mid  f_i:S\to K_\lambda,f_i\in\Omega \textrm{ for each }i\in\{1,\ldots,m\}\}\cong K_\lambda^{|\Delta|m}.$$

Consider the subset $C$ of $A$ defined by 
\begin{eqnarray}\label{AC}
C&=&\{(g,f_1,\ldots,f_m)\in A\mid g\in L_\lambda\}.
\end{eqnarray}
Since $\varphi$ is a homomorphism, $C$ is a subgroup of $A$. 

\begin{lemma}\label{silly4}$M$ is the core of $C$ in $A$ and $M\cong K_\lambda^{|\Delta|m}$. The action of $A$ on the right cosets of $C$ is permutation isomorphic to the action of $L$ on $\Lambda$.
\end{lemma}
\begin{proof}
From Lemma~\ref{silly3}, the map $\varphi$ is a surjective homomorphism and hence $A/M\cong L$. Furthermore, $C^{\varphi}=L_\lambda$. 
As $M\leq C\leq A$, $M\unlhd A$ and $L_\lambda$ is core-free in $L$, we see that $M$ is the core of $C$ in $A$. Finally, the action of $A$ on the right cosets of $C$ is permutation isomorphic to the action of $L=A^{\varphi}$ on the right cosets of $L_\lambda=C^{\varphi}$, that is, permutation isomorphic to the action of $L$ on $\Lambda$.
\end{proof} 

Assume that $m$ is odd. Let $c=(g,f_1,\ldots,f_m)\in C$ and set
\begin{equation*}
c^\iota=(g,f_1,\ldots,f_m)^\iota=(f_1(1),f_1^{\iota_{1,c}},\ldots,f_m^{\iota_{m,c}})
\end{equation*}
where
\begin{eqnarray}\label{iota}
f_i^{\iota_{i,c}}(x)&=&\left\{
\begin{array}{lcl}\label{iota2}
g&&\textrm{if }x\in S_\delta \textrm{ and }i=1,\\
f_{i+1}(x)&&\textrm{if }x\in S_\delta, 2\leq i\leq m-1 \textrm{ and }i \textrm{ is even},\\
f_{i-1}(x)&&\textrm{if }x\in S_\delta, 3\leq i\leq m \textrm{ and }i \textrm{ is odd},\\
f_{i-1}(x)&&\textrm{if }x\notin S_\delta, 2\leq i\leq m-1 \textrm{ and }i \textrm{ is even},\\
f_{i+1}(x)&&\textrm{if }x\notin S_\delta, 1\leq i\leq m-2 \textrm{ and }i \textrm{ is odd},\\
f_m(x)&&\textrm{if }x\notin S_\delta \textrm{ and }i=m.
\end{array}
\right.
\end{eqnarray}
Note that, for each $i\in \{1,\ldots,m\}$, the definition of $\iota_{i,c}$ depends on $i$ and on the element $c$. In particular, to compute $f_i^{\iota_{i,c}}$ one needs to know $i$ and the coordinates of  $c$.

\begin{lemma}\label{silly5}The map $\iota$ is an automorphism of $C$ with $\iota^2=1$. 
\end{lemma}
\begin{proof}
Let $c=(g,f_1,\ldots,f_m)\in C$. We have $g^\pi\in L_\lambda^\pi=S_\delta$ and hence $(g^\pi)^\tau=((g^\pi)^{-1})^\tau=1$. Therefore, for each $i\in \{1,\ldots,m\}$, we obtain  $$(\dag)\qquad f_i(1)^\pi=((g^\pi)^{-1})^{\tau}g^\pi(1^\tau)^{-1}=g^\pi.$$ In particular, $f_1(1)\in L_\lambda$ and $f_1(1)^\pi=g^\pi$.

We first show that $c^\iota\in C$. From~$(\ref{iota2})$, we see that, for each $i\in \{1,\ldots,m\}$ and for each $x\in S$, the function $f_i^{\iota_{i,c}}$ is constant on the right coset $S_\delta x$ of $S_\delta$ in $S$ and therefore $f_i^{\iota_{i,c}}\in \Omega$.  Let $x\in S$. We have to show that, for each $i\in \{1,\ldots,m\}$, we have $(f_i^{\iota_{i,c}}(x))^\pi=(x(f_1(1)^\pi)^{-1})^\tau f_1(1)^\pi(x^\tau)^{-1}$. In particular, as $f_1(1)^\pi=g^\pi$, we have to show that $(f_i^{\iota_{i,c}}(x))^\pi=(x(g^\pi)^{-1})^\tau g^\pi (x^\tau)^{-1}$. Since $g^\pi\in S_\delta$ and $(g,f_1,\ldots,f_m)\in A$, this is clear from the definition of $f_i^{\iota_{i,c}}$, except possibly when $i=1$ and $x\in S_\delta$. Hence, we assume that $i=1$ and $x\in S_\delta$. In particular, $(x(g^\pi)^{-1})^\tau=x^\tau=1$. Thus, we have $$(f_1^{\iota_{1,c}}(x))^\pi=g^\pi=(x(g^\pi)^{-1})^\tau g^\pi(x^\tau)^{-1}.$$
Therefore, $c^\iota\in C$. 

Now we show that  $\iota^2$ is the identity permutation of $C$. We have

\begin{eqnarray*}
c^{\iota^2}&=&(f_1(1),f_1^{\iota_{1,c}},\ldots,f_m^{\iota_{m,c}})^\iota=(f_1^{\iota_{1,c}}(1),(f_1^{\iota_{1,c}})^{\iota_{1,c^\iota}},\ldots,(f_m^{\iota_{m,c}})^{\iota_{m,c^\iota}}).
\end{eqnarray*}
Now, from~$(\ref{iota})$ we have $f_1^{\iota_{1,c}}(1)=g$ and hence the first coordinate of $c^{\iota^2}$ equals the first coordinate of $c$. Furthermore, \[
(f_1^{\iota_{1,c}})^{\iota_{1,c^\iota}}=
\left\{
\begin{array}{ll}
f_1(x)&\textrm{if }x\in S_\delta\\
f_2^{\iota_{1,c}}(x)=f_1(x)&\textrm{if }x\notin S_\delta\\
\end{array}
\right.
\] 
and hence the second coordinate of $c^{\iota^2}$ equals the second coordinate of $c$. Assume that $2\leq i\leq m-1$ is even. We have
\[
(f_i^{\iota_{1,c}})^{\iota_{1,c^\iota}}=
\left\{
\begin{array}{ll}
f_{i+1}^{\iota_{1,c}}(x)=f_i(x)&\textrm{if }x\in S_\delta\\
f_{i-1}^{\iota_{1,c}}(x)=f_i(x)&\textrm{if }x\notin S_\delta\\
\end{array}
\right.
\] 
and hence the $(i+1)$th coordinate of $c^{\iota^2}$ equals the $(i+1)$th coordinate of $c$. The proof that $(f_{i}^{\iota_{i,c}})^{\iota_{i,c^\iota}}=f_i$ when $3\leq i\leq m$ and $i$ is odd is very similar to the previous case and is left to the conscious reader.

Finally, we show that $\iota$ is an automorphism of $C$. Let $e:S\to L_\lambda$ such that $e(x)=1$ for every $x\in S$. The element $(1,e,\ldots,e)$ is the identity of $C$ and $(1,e,\ldots,e)^\iota=(1,e,\ldots,e)$. Now, let $c=(g,f_1,\ldots,f_m)$ and $c'=(g',h_1,\ldots,h_m)$ be in $C$. We have

\begin{eqnarray*}
(cc')^\iota&=&(gg',f_1^{g'}h_1,\ldots,f_m^{g'}h_m)^\iota\\
&=&((f_1^{g'}h_1)(1),(f_1^{g'}h_1)^{\iota_{1,cc'}},\ldots,(f_m^{g'}h_m)^{\iota_{m,cc'}})\\
\end{eqnarray*}
and 
\begin{eqnarray*}
c^\iota c'^\iota&=&(f_1(1),f_1^{\iota_{1,c}},\ldots,f_m^{\iota_{m,c}})(h_1(1),h_1^{\iota_{1,c'}},\ldots,h_m^{\iota_{m,c'}})\\
&=&(f_1(1)h_1(1),(f_1^{\iota_{1,c}})^{h_1(1)}h_1^{\iota_{1,c'}},\ldots,(f_m^{\iota_{m,c}})^{h_1(1)}h_m^{\iota_{m,c'}}).\\
\end{eqnarray*}
Recall that from~$(\dag)$, we have $h_1(1)^\pi=g'^\pi$. As $g'^\pi\in S_\delta$ and $f_1$ is constant on $S_\delta$, we obtain $(f_1^{g'}h_1)(1)=f_1((g'^\pi)^{-1})h_1(1)=f_1(1)h_1(1)$ and hence the first coordinate of $(cc')^\iota$ equals the first coordinate of $c^\iota c'^\iota$. We have
\[
(f_1^{g'}h_1)^{\iota_{1,cc'}}(x)=\left\{
\begin{array}{ll}
gg'&\textrm{if }x\in S_\delta\\
(f_2^{g'}h_2)(x)&\textrm{if }x\notin S_\delta\\
\end{array}
\right.\]
and
\[
((f_1^{\iota_{1,c}})^{h_1(1)}h_1^{\iota_{1,c'}})(x)=\left\{
\begin{array}{llll}
f_1^{\iota_{1,c}}(x(h_1(1)^\pi)^{-1})h_1^{\iota_{1,c'}}(x)&=&gg'&\textrm{if }x\in S_\delta\\
f_1^{\iota_{1,c}}(x(h_1(1)^\pi)^{-1})h_2(x)&=&f_2(x(g'^\pi)^{-1})h_2(x)&\\
&=&(f_2^{g'}h_2)(x)&\textrm{if }x\notin S_\delta.\\
\end{array}
\right.
\]
Therefore $(f_1^{g'}h_1)^{\iota_{1,cc'}}=(f_1^{\iota_{1,c}})^{h_1(1)}h_1^{\iota_{1,c'}}$ and hence the second coordinate of $(cc')^\iota$ equals the second coordinate of $c^\iota c'^\iota$. Assume that $2\leq i\leq m-1$ is even.  We have
\[
(f_i^{g'}h_i)^{\iota_{i,cc'}}(x)=\left\{
\begin{array}{ll}
(f_{i+1}^{g'}h_{i+1})(x)&\textrm{if }x\in S_\delta\\
(f_{i-1}^{g'}h_{i-1})(x)&\textrm{if }x\notin S_\delta\\
\end{array}
\right.\]
and $((f_i^{\iota_{i,c}})^{h_1(1)}h_i^{\iota_{i,c'}})(x)$ equals
\[
\left\{
\begin{array}{ll}
f_i^{\iota_{i,c}}(x(h_1(1)^\pi)^{-1})h_{i+1}(x)=f_{i+1}(x(g'^\pi)^{-1})h_{i+1}(x)=(f_{i+1}^{g'}h_{i+1})(x)&\textrm{if }x\in S_\delta\\
f_i^{\iota_{i,c}}(x(h_1(1)^\pi)^{-1})h_{i-1}(x)=f_{i-1}(x(g'^\pi)^{-1})h_{i-1}(x)=(f_{i-1}^{g'}h_{i-1})(x)&\textrm{if }x\notin S_\delta.\\
\end{array}\right.
\]
Therefore $(f_{i}^{g'}h_i)^{\iota_{i,cc'}}=(f_i^{\iota_{i,c}})^{h_1(1)}h_i^{\iota_{i,c'}}$. The proof that $(f_{i}^{g'}h_i)^{\iota_{i,cc'}}=(f_i^{\iota_{i,c}})^{h_1(1)}h_i^{\iota_{i,c'}}$ when $3\leq i\leq m$ and $i$ is odd is very similar to the previous case and is left to the conscious reader.
\end{proof}

Define
\begin{equation}\label{B}
B=C\rtimes \langle\iota\rangle.
\end{equation}

\begin{lemma}\label{silly6}
If $N$ is a subgroup of $C$ with $N\unlhd A$ and $N\unlhd B$, then $N=1$.
\end{lemma}
\begin{proof}
From Lemma~\ref{silly4}, the group $M$ is the core of $C$ in $A$. Since $N\unlhd A$ and $N\leq C$, we get $N\leq M$ and hence every element of $N$ is of the form $(1,f_1,\ldots,f_m)$ where $f_1,\ldots,f_m\in \Omega$. 

We first prove the following preliminary claim.

\smallskip

\noindent\textsc{Claim. }Let $i\in \{1,\ldots.m\}$ and let $s'\in S$. Assume that, for every element $(1,f_1,\ldots,f_m)$ of $N$, we have $f_i(s')=1$. Then $f_i(s)=1$ for every $s\in S$. 

\smallskip

\noindent 
Let $(1,f_1,\ldots,f_m)\in N$. Fix $g$ an arbitrary element of $L$. From Lemma~\ref{silly3}, there exist $h_1,\ldots,h_m\in \Omega$ with $(g,h_1,\ldots,h_m)\in A$. As  $N\unlhd A$, 
$$
(1,f_1,\ldots,f_m)^{(g,h_1,\ldots,h_m)}=
(1,h_1^{-1}f_1^gh_1,\ldots,h_m^{-1}f_m^{g}h_m)\in N.$$ 
By hypothesis, 
$1=(h_i^{-1}f_i^{g}h_i)(s')=
h_i(s')^{-1}f_i(s'(g^\pi)^{-1})h_i(s')$ and hence 
$f_i(s'(g^\pi)^{-1})=1$. Since $g$ is an arbitrary element of $L$ and $\pi:L\to S$ is surjective, we obtain $f_i=1$ .~$_\blacksquare$

\smallskip

We argue by contradiction and we assume that $N\neq 1$. Let $j$ be the minimal element of $\{1,\ldots,m\}$ such that $N$ contains a non-identity element $(1,f_1,\ldots,f_m)$ with $f_j\neq 1$. Let $x=(1,f_1,\ldots,f_m)$ be an arbitrary element of $N$.  Since $N\unlhd B$ and $\iota\in B$, we have that $$(\dag)\qquad (1,f_1,\ldots,f_m)^\iota=(f_1(1),f_1^{\iota_{1,x}},\ldots,f_m^{\iota_{m,x}})$$ lies in $N$ and hence $f_1(1)=1$. Since $x$ is an arbitrary element of $N$, by applying Claim with $i=1$ and $s'=1$, we get $f_1=1$ for every element $(1,f_1,\ldots,f_m)$ of $N$. This proves that $j>1$. Assume that $j$ is even. Let $s'\in S\setminus S_\delta$. From~$(\ref{iota})$ and from the minimality of $j$, we see that $f_j^{\iota_{j,x}}(s')=f_{j-1}(s')=1$. Since $x^\iota$ is an arbitrary element of $N^\iota=N$, by applying Claim with $i=j$, we get $f_j=1$ for every element $(1,f_1,\ldots,f_m)$ of $N$, contradicting the choice of $j$. Assume that $j$ is odd. 
From~$(\ref{iota})$ and from the minimality of $j$, we see that $f_j^{\iota_{j,x}}(1)=f_{j-1}(1)=1$. Since $x^\iota$ is an arbitrary element of 
$N^\iota=N$, by applying Claim with $i=j$ and $s'=1$,  we have $f_j=1$ for every 
element $(1,f_1,\ldots,f_m)$ of $N$, again contradicting the minimality of $j$. 
This last contradiction concludes the proof.
\end{proof}

\begin{remark}{\rm
The definition of $\iota$ in~$(\ref{iota})$ depends on the fact that we have chosen $m$ odd. Nevertheless, when $m$ is even, it is possible to define (in a similar fashion) an involutory automorphism of $C$ satisfying Lemma~\ref{silly6}.}
\end{remark}

Now we recall the definition of \emph{coset graph}. For a group $G$, a subgroup $A$ and an element $b \in G $, the coset graph $\Cos(G, A, b)$ is the graph with vertex set the set of right cosets $G/A = \{Ag \mid g \in G\}$ and edge set $\{\{Ag, Abg\} \mid g \in G\}$. The following proposition is due to Sabidussi~\cite{sabidussi}.

\begin{proposition}\label{sab}
Let $A$ be a core-free subgroup of $G$ and let $b\in G$ with $G=\langle A,b\rangle$ and $b^{-1}\in AbA$. Then $\Gamma=\Cos(G,A,b)$ is a connected $G$-arc-transitive graph and the action of the stabiliser $G_v$ of the vertex $v$ of $\Gamma$ on $\Gamma(v)$ is permutation isomorphic to the action of $A$ on the right cosets of $A\cap A^b$ in $A$. 
\end{proposition}

The following lemma is part of folklore but we include a proof for sake of completeness.

\begin{lemma}\label{silly7}
Let $A$, $B$ and $C$ be finite groups with $C=A\cap B$ and $|B:C|=2$. Assume that $1$ is the only subgroup of $C$ normal both in $A$ and in $B$. Then there exists a finite group $\bar{G}$ and a connected $\bar{G}$-arc-transitive graph $\Gamma$ such that the stabiliser $\bar{G}_v$ of the vertex $v$ of $\Gamma$ is isomorphic to $A$ and the action of $\bar{G}_v$ on $\Gamma(v)$ is permutation isomorphic to the action of $A$ on the right cosets of $C$ in $A$.
\end{lemma}
\begin{proof}
Let $G=A\ast_CB$ be the free product of $A$ with $B$ amalgamated over $C$. We identify $A$, $B$ and $C$ with their  corresponding isomorphic copies in $G$. Fix $b\in B\setminus C$. Since $G$ is the free product of $A$ and $B$ and since $b$ normalises the subgroup $C$ of $B$, we have $C=A\cap A^b$. It is shown in~\cite[Theorem~$2$]{Baum} that $G$ is a residually finite group.
As $B$ is a finite group and $AA^b=\{xy^b\mid x,y\in A\}$ is a finite set, $G$ contains a normal subgroup of finite index $N$ with $B\cap N=1$ and $AA^b\cap N=1$. Let $\bar{G}=G/N$ and denote by $-:G\to \bar{G}$ the natural projection (in the rest of the proof we use the bar convention, that is, we denote by $\bar{X}$ the image of $X$ under $-$).  

Since $G=\langle A,B\rangle=\langle A,b\rangle$, we have $\bar{G}=\langle \bar{A},\bar{b}\rangle$. Clearly, $\bar{C}=\overline{A\cap A^b}\leq \bar{A}\cap \bar{A}^{\bar{b}}$. Let $\bar{x}\in \bar{A}\cap \bar{A}^{\bar{b}}$. We have $x\in AN\cap A^bN$ and hence $x=a_1n_1=a_2^bn_2$ for some $a_1,a_2\in A$ and $n_1,n_2\in N$. Thus $a_1^{-1}a_2^b=n_1n_2^{-1}\in AA^b\cap N=1$, $a_1=a_2^b$ and $x\in (A\cap A^b)N=CN$. It follows that $\bar{x}\in \bar{C}$. This shows that $\bar{A}\cap \bar{A}^{\bar{b}}=\bar{C}$. Let $\bar{U}=U/N$ be a normal subgroup of $\bar{G}$ contained in $\bar{A}$. We have $\bar{U}\leq \bar{A}\cap \bar{A}^{\bar{b}}=\bar{C}$. Therefore $U$ is a subgroup of $C$ normal both in  $A$ and in $B$. It follows that $U=1$, $\bar{U}=1$ and $\bar{A}$ is a core-free subgroup of $\bar{G}$.  Since $|B:C|=2$, we have $b^{-1}\in Cb\subseteq Ab$ and $\bar{b}^{-1}\in \bar{A}\bar{b}\subseteq \bar{A}\bar{b}\bar{A}$. Since $A\cap N=1$, the restriction of $-$ to $A$ is an isomorphism fr
 om $A$ to $\bar{A}$ mapping $C$ to $\bar{C}$. Therefore the action of $A$ on the right cosets of $A\cap A^b=C$ in $A$ is permutation isomorphic to the action of $\bar{A}$ on the right cosets of $\bar{A}\cap \bar{A}^{\bar{b}}$ in $\bar{A}$. The conclusion then follows from Proposition~\ref{sab} applied to $\bar{G},\bar{A}$ and $\bar{b}$.
\end{proof}

Finally, we are ready to prove the main theorem of this section.\\

\noindent\textit{Proof of Theorem~$\ref{Theo:ResAreSP}$.} Let $L$ be graph-restrictive. Suppose, by contradiction, that $L$ is not semiprimitive and let $K$ be an intransitive normal subgroup of $L$ which is not semiregular. Let $m$ be an arbitrary odd positive integer, let $A$ be as in~$(\ref{A})$, let $C$ be as in~$(\ref{AC})$ and let $B$ be as in~$(\ref{B})$. From Lemma~\ref{silly6}, we see that we may apply Lemma~\ref{silly7} to $A$, $B$ and $C$. Therefore, there exists a $G$-arc-transitive graph $\Gamma$ with the stabiliser $G_v$ of the vertex $v$  isomorphic to $A$ and with the action of $G_v$ on $\Gamma(v)$ permutation isomorphic to the action of $A$ on the right cosets of $C$, which, by Lemma~\ref{silly4}, is equivalent to the action of $L$ on $\Lambda$.  Thus $(\Gamma,G)$ is locally-$L$. Now, using Lemma~\ref{silly4} again, we obtain that the kernel of the action of $G_v$ on $\Gamma(v)$ is isomorphic to $K_\lambda^{|\Delta|m}$.  Since $K$ is not semiregular, we have 
 $|K_\lambda|\neq 1$. Furthermore, as $m$ is an arbitrary odd integer, we have that the size of $G_v^{[1]}$ cannot be bounded above by a function of $L$. Thus $L$ is not graph-restrictive. \hfill\qedsymbol

\section{$p$-graph-restrictive groups}
\label{p-restrictive}
In this section using the following theorem of Glauberman we prove Corollary~\ref{Cor:p-restrictive}.

\begin{theo} \label{Theo:Glaub}\cite[Theorem 1]{Glaub}
Suppose $P$ is a subgroup of a finite group $G$, $g\in G$, and $P\cap P^g$ is a normal subgroup of prime index $p$ in $P^g$. Let $n$ be a positive integer, and let $\tilde{G}=\langle P,P^g,\ldots,P^{g^n}\rangle$. Assume that :
\begin{enumerate}
\item $g$ normalises no non-identity normal subgroup of $P$, and
\item $P\cap Z(\tilde{G})=1$.
\end{enumerate}
Then $|P|=p^t$ for some positive integer $t$ for which $t\leq 3n$ and $t\neq 3n-1$. Moreover, if $n=2$, $p=2$ and $t=6$, then $P$ contains a non-identity normal subgroup of $\tilde{G}$.
\end{theo}

\noindent\textit{Proof of Corollary~$\ref{Cor:p-restrictive}$. }
Let $p$ be a prime and let $L$ be a transitive permutation group on $\Omega$. Let $x\in\Omega$ and let $\bar{P}$ be a Sylow $p$-subgroup of $L_x$ such that $|\bar{P}|=p$, and there exists ${\bar{l}}\in L$ such that $\langle \bar{P},\bar{P}^{\bar{l}}\rangle$ is transitive on $\Omega$. We must show that $L$ is $p$-graph-restrictive.

Let $(\Gamma,G)$ be a locally-$L$ pair, let $(u,v)$ be an arc of $\Gamma$ and let $K=G_v^{[1]}$ be the kernel of the action of $G_v$ on the neighbourhood of $v$. Let $P$ be a Sylow $p$-subgroup of $G_{uv}$. By hypothesis, we have $|P:K\cap P|=p$ and, moreover, there exists $l\in G_v$ with $\langle P,P^l\rangle$ transitive on $\Gamma(v)$. 

Let $R$ be a Sylow $p$-subgroup of $KP^l$ containing $Q=K\cap P$. Note that $\langle P,R\rangle$ is transitive on $\Gamma(v)$ and hence $P\neq R$. Let $w=u^l\in\Gamma(v)$ and note that $KP^l\leq (G_{uv})^l=G_{wv}$ and hence $R$ is a Sylow $p$-subgroup of $G_{wv}$. It follows that $Q$ has index $p$ in both $P$ and $R$ and, in particular, is normal in both of them. Since $P\neq R$, we have that $Q=P\cap R$.

Since $\Gamma$ is $G$-arc-transitive, there exists $\sigma\in G$ such that $(u,v)^\sigma=(v,w)$. Since both $P^\sigma$ and $R$ are Sylow $p$-subgroups of $G_{wv}$, there exist $h\in G_{wv}$ such that $R=P^{\sigma h}$. Writing $g=\sigma h$, we get $R=P^g$ and $(u,v)^g=(v,w)$. As $\langle P,P^g\rangle=\langle P,R\rangle$ is transitive on $\Gamma(v)$, it follows that $\langle P^g,P^{g^2}\rangle$ is transitive on $\Gamma(w)$. Hence $\tilde{G}=\langle P,P^g,P^{g^2}\rangle$ is transitive on the edges of $\Gamma$ and $\langle g,P\rangle$ is transitive on the arcs of $\Gamma$. Since $P\leq G_{uv}$, the element $g$ normalises no non-identity normal subgroup of $P$, $P$ contains no non-trivial normal subgroup of $\tilde{G}$ and $P\cap\Z(\tilde{G})=1$. We can then use Theorem~\ref{Theo:Glaub} with $n=2$ to conclude.\hfill\qedsymbol

\section{Restrictiveness of $2$-transitive groups}
\label{TwoTransitive}
In this section we prove Corollary~\ref{TechnicalPLocal2} and we outline the proof of Theorem~\ref{TwoTransitiveTheo}. Let $\Gamma$ be a graph, let $v$ be a vertex of $\Gamma$ and let $G\leq\Aut(\Gamma)$. The subgroup of $G_v$ fixing the neighbourhood $\Gamma(v)$ of the vertex $v$ point-wise will be denoted by $G_v^{[1]}$. As before, we will also write $G_{uv}^{[1]}$ for $G_u^{[1]}\cap G_v^{[1]}$. Finally, $G_{\{u,v\}}$ will denote the set-wise stabiliser of the edge $\{u,v\}$.

\begin{lemma}\label{TechnicalPLocalLemma}
Let $(\Gamma,G)$ be a  locally-transitive pair and let $(u,v)$ be an arc of $\Gamma$. If $G_{uv}^{[1]}$ is a non-trivial $p$-group, then $\OO_p(G_{uv}^{\Gamma(u)})\neq 1$.
\end{lemma}
\begin{proof}
Since $G_v^{[1]}\unlhd G_{uv}$, we have that $\OO_p(G_v^{[1]})\unlhd G_{uv}$ and hence $\OO_p(G_v^{[1]})\leq\OO_p(G_{uv})$. Suppose that $\OO_p(G_{uv})\leq G_v^{[1]}$.  Then $\OO_p(G_{uv})\leq \OO_p(G_v^{[1]})$, which implies $\OO_p(G_{uv})=\OO_p(G_v^{[1]})$ and therefore $\OO_p(G_{uv})\unlhd G_v$. It follows that $\OO_p(G_{uv})\unlhd\langle G_v,G_{\{u,v\}}\rangle=G$ and hence $\OO_p(G_{uv})=1$, which is a contradiction. We may thus assume that $\OO_p(G_{uv})\nleq G_u^{[1]}$ and, in particular, $1\neq \OO_p(G_{uv})^{\Gamma(u)}\leq \OO_p(G_{uv}^{\Gamma(u)})$.
\end{proof}

\noindent\textit{Proof of Corollary~$\ref{TechnicalPLocal2}$. }We argue by contradiction and we assume that $L$ is not $p$-graph-restrictive. By Theorem~\ref{LocalSpiga}, there exists a locally-$L$ pair $(\Gamma,G)$ and an arc $(u,v)$ of $\Gamma$ such that $G_{uv}^{[1]}$ is a non-trivial $p$-group.  By Lemma~\ref{TechnicalPLocalLemma}, we have $\OO_p(L_v)\cong\OO_p(G_{uv}^{\Gamma(u)})\neq 1$, a contradiction.\hfill\qedsymbol

\medskip

Corollary~\ref{TechnicalPLocal2} was previously known in the case when $L$ is primitive. It is a very strong restriction, especially when $L$ is $2$-transitive. As a corollary of the Classification Theorem for Finite Simple Groups, the  $2$-transitive  groups are known, see for example~\cite[Theorem~$5.3$]{Peter}. In~\cite[Sections~$7.4$,~$7.5$]{Peter2}, the reader can find the complete list of $2$-transitive groups (\cite[Table~$7.2$]{Peter2} describes all possible socles of an almost simple $2$-transitive group and~\cite[Table~$7.3$]{Peter2} describes all possible point-stabilisers of an affine $2$-transitive group). It is then simply a matter of going through these tables and a little work to obtain the next result.

\begin{theo}\label{2TransPLocal}
Let $p$ be a prime and let $L$ be a $2$-transitive permutation group on a finite set $\Omega$ and $x\in\Omega$. If $\OO_p(L_x)\neq 1$, then either $L$ is an affine $2$-transitive group, or the socle $T$ of $L$ is $2$-transitive and isomorphic to one of $\PSL(n,q)$, $\PSU(3,q)$, $\Suz(q)$, or $\Ree(q)$ acting in their natural $2$-transitive action, or $L=\PGammaL(2,8)$ of degree $28$.
\end{theo}

By our observation preceding Theorem~\ref{2TransPLocal} and from~\cite[Theorem]{weissp}, to complete the proof of Theorem~\ref{TwoTransitiveTheo}, it suffices to show that $L$ is graph-restrictive when $L$ is an almost simple $2$-transitive group with socle $T$ as in Theorem~\ref{2TransPLocal}. In~\cite{weissu}, Weiss provided a unified proof that $L$ is graph-restrictive in the case where $T$ is isomorphic to $\PSU(3,q)$, $\Suz(q)$, or $\Ree(q)$ (recall that $\PGammaL(2,8)$ acting on $28$ points can be thought of as the non-simple group $\Ree(3)$ in its natural $2$-transitive action). This had been proved earlier by various authors, using more ad hoc methods. The case where $T$ is isomorphic to $\PSL(n,q)$ turned out to be the hardest to deal with, by far. The case where $n=2$ was dealt with by Weiss in \cite{weissBN}. The proof in the general case was first announced in \cite{trofORIG}, but only a brief sketch of the proof was given there. 
 The proof in the case where $n\ge 3$ and $p$ is odd is scattered over several papers and a complete overview of these results is given in the introductory section of \cite{trof1}. Finally, the case of characteristic $2$ was treated in a series of four paper
\cite{trof1,trof2,trof3,trof4}, thus completing the proof of Theorem~\ref{TwoTransitiveTheo}.


\end{document}